\documentclass[12pt]{amsart}

\usepackage{amssymb}
\usepackage{amsmath}
\usepackage{amsthm}
\usepackage{mathtools}
\usepackage{enumerate}
\usepackage{dsfont}
\usepackage[title]{appendix}
\usepackage{comment}
\usepackage[usenames,dvipsnames]{color}

\theoremstyle{plain}

\newtheorem{theorem}{Theorem}[section]

\newtheorem{lemma}[theorem]{Lemma}
\newtheorem{proposition}[theorem]{Proposition}
\newtheorem{corollary}[theorem]{Corollary}

\newtheorem{observation}[theorem]{Observation}

\newtheorem{problem}[theorem]{Problem}

\theoremstyle{definition}

\newtheorem{definition}[theorem]{Definition}
\newtheorem{remark}[theorem]{Remark}

\newcommand{\cs}{cone smooth}

\newcommand{\supp}{\mathrm{supp}}

\newcommand{\dist}{\mathrm{dist}}
\newcommand{\C}{\mathrm{C}}

\newcommand{\SX}{S_{X^*}}

\newcommand{\Q}{\mathbb{Q}}
\newcommand{\R}{\mathbb{R}}
\newcommand{\N}{\mathbb{N}}

\newcommand{\K}{\mathcal{K}}

\renewcommand{\epsilon}{\varepsilon}
\renewcommand{\phi}{\varphi}
\renewcommand{\tilde}{\widetilde}

\newcommand{\B}{\mathcal{B}}
\newcommand{\F}{\mathcal{F}}

\def\norm#1{\|#1\|}

\def\norma#1{|\!|\!|#1|\!|\!|}

\DeclareMathOperator{\dens}{dens}
\DeclareMathOperator{\diam}{diam}

\newcommand{\vertiii}[1]{{\left\vert\kern-0.25ex\left\vert\kern-0.25ex\left\vert #1 
    \right\vert\kern-0.25ex\right\vert\kern-0.25ex\right\vert}}

\begin{document}

\title{Star-finite coverings of Banach spaces}

\author[C.A.~De~Bernardi]{Carlo Alberto De Bernardi}
\author[J.~Somaglia]{Jacopo Somaglia}
\author[L.~Vesel\'y]{Libor Vesel\'y}

\address{Dipartimento di Matematica per le Scienze economiche, finanziarie ed attuariali, Universit\`a Cattolica del Sacro Cuore, 20123 Milano,Italy}

\email{carloalberto.debernardi@unicatt.it}

\address{Dipartimento di Matematica ``F. Enriques'', Universit\`a degli Studi di Milano, via Cesare Saldini 50, 20133 Milano, Italy. }

\email{jacopo.somaglia@unimi.it}
\email{libor.vesely@unimi.it}



 \thanks{ All the authors are members of GNAMPA (INdAM). 
The research of the first and the third author was partially supported by GNAMPA (INdAM), Project 2018. The research of the second and the third author was partially supported by the University of Milan (Universit\`a degli Studi di Milano), Research Support Plan 2019.}

\begin{abstract} 
We study star-finite coverings of infinite-dimensional normed spaces. A family of sets is called star-finite if each of its members intersects only finitely many other members of the family. It follows by our results that an LUR or a uniformly Fr\'echet smooth infinite-dimensional Banach space does not admit star-finite coverings by closed balls. On the other hand, we present a quite involved construction proving existence of a star-finite covering of $c_0(\Gamma)$  by Fr\'echet smooth centrally symmetric bounded convex bodies.  A similar but simpler construction shows that every normed space of countable dimension (and hence incomplete) has a star-finite covering by closed balls.

\end{abstract}

\maketitle

\section{Introduction}

A family of subsets of a real normed space $X$ is called a {\em covering} if the union of all its members coincides with $X$.
One of the earliest results concerning coverings of infinite-dimensional spaces is Corson's theorem \cite{Cor61}, stating that if  $X$ is a reflexive infinite-dimensional Banach space and $\F$ is a covering of $X$ by bounded convex sets  then $\F$ is not locally finite (see Definition~\ref{D: finiteness}). V.P.~Fonf and C.~Zanco \cite{FZ06} improved this result by proving that {\em if a Banach space $X$ contains an infinite-dimensional closed subspace non containing $c_0$ then $X$ does not admit any locally finite covering by bounded closed convex bodies}. The same authors proved  in \cite{FZJMMA09} that  {\em if $X$ contains a separable infinite-dimensional dual space and if $\tau$ is a covering by bounded closed convex sets then there exists a finite-dimensional compact  set $C$ that meets infinitely many members of $\tau$}. Moreover, they proved in \cite{FZForum09} that, in the above result, if the members of $\tau$ are rotund or smooth then $C$ can be taken 1-dimensional. Let us recall that the prototype of a locally finite covering of an infinite-dimensional Banach space by closed convex bounded sets  is the covering (actually a tiling) of $c_0$ by translates of its unit ball, see \cite{lup88} for the details. (Recall that a {\em tiling} is a covering by bodies whose nonempty interiors are pairwise disjoint.)

The existing theory of point-finite coverings (see Definition~\ref{D: finiteness}) of infinite-dimen\-sional normed spaces is less developed and mainly concerns coverings by balls. 
A surprising construction discovered in 1981 by V.~Klee \cite{Klee1} shows 
existence of a simple (that is, disjoint, and hence point-finite) covering  of $\ell_1(\Gamma)$ by  closed balls of radius $1$, whenever $\Gamma$ is a suitable uncountable set. 
Though  the question about existence of point-finite coverings by balls of $\ell_p(\Gamma)$ spaces 
 was already considered by V.~Klee in the same paper, this problem was partially solved only  recently  for $\ell_2$ by
 V.P.~Fonf and C.~Zanco in \cite{FZHilbert}, in which they proved that {\em the infinite-dimensional separable Hilbert space does not admit point-finite coverings by closed balls of positive radius}. 
 Then V.P.~Fonf, M.~Levin and C.~Zanco \cite{FonfLevZan14} extended this  result to separable Banach spaces that are both uniformly smooth and uniformly rotund. We point out that Klee's  problem about coverings by closed balls seems to be  open in the non-separable case, even for Hilbert spaces. 

In the present paper, we consider a particular class of coverings of infinite-dimensional normed spaces, given by the property that each member intersects at most finitely many other members. Such coverings are known in the literature
as {\em star-finite coverings} (see \cite[p.~317]{Engel}), and singular points of star-finite (not necessarily convex) tilings of topological vector spaces were first studied in \cite{Breen85}, then generalized in  \cite{NielsenStarfinite}.
It is clear that each simple covering is star-finite and each star-finite covering is point-finite.
Moreover, the above-mentioned coverings by balls of $c_0$ and $\ell_1(\Gamma)$ easily show that there are no implications between star-finiteness and local finiteness of a covering.

Roughly speaking, all mentioned results concerning non-existence of point-finite or locally finite coverings are in some sense inspired by the following general principle.

\smallskip
  \noindent {\em Coverings in ``good'' (separable, reflexive, \ldots) infinite-dimensional Banach spaces whose  members  enjoy ``nice properties'' (smoothness, rotundity, \ldots) cannot satisfy  ``finiteness properties'' (local finiteness, point finiteness, \ldots).}
  
\smallskip
\noindent
 Hence, the first step in our study is to determine to what extent we can apply the same principle to star-finite coverings.
A careful reading of the proof of a result by A.~Marchese and C.~Zanco \cite{MarZan}, stating that  each Banach space has a 2-finite (see Definition~\ref{D: finiteness})  covering (actually a tiling)  by closed convex bounded bodies, reveals that the same argument actually proves that  {\em each Banach space admits a  covering  by closed convex bounded bodies such that each of its member intersects at most two other its members}. However, as noted by the authors, the elements of such a covering are far from being balls.  This fact together with Klee's construction in $\ell_1(\Gamma)$ suggest that, in order to obtain non-existence results, we should
restrict at first our attention to star-finite coverings by closed balls satisfying some rotundity or smoothness property. After some preliminaries and some general facts (Section~\ref{section:preliminaries}), we prove the main results in this direction in Section~\ref{section:prohibitive}: our Corollary~\ref{C: suffcond} implies that an infinite-dimensional Banach space $X$ does not admit any star-finite covering by closed balls whenever $X$  is uniformly Fr\'echet smooth or LUR. The techniques used in some of these proofs are inspired by the paper \cite{DEVETIL}. We also prove non existence of {\em countable} star-finite coverings by closed balls for a class of (subspaces of) spaces of continuous functions, which include, e.g., all infinite-dimensional $\ell_\infty(\Gamma)$ spaces. In the particular case of $c_0(\Gamma)$ ($\Gamma$ infinite), we show that it admits no (countable or not) star-finite covering by closed balls.

In Section~\ref{section:existenceresult}, we obtain a result in the opposite direction: we present a quite involved construction of a star-finite covering of every $c_0(\Gamma)$ space by Fr\'echet smooth centrally symmetric bounded bodies. The starting point of our construction is existence of an equivalent Fr\'echet smooth norm on $c_0(\Gamma)$ whose unit sphere contains many ``flat faces''
(see Proposition~\ref{P: bollapiatta}). We point out that a similar but simpler construction contained in Section~\ref{section:preliminaries} shows that every normed (necessarily incomplete) space of countable dimension has a star-finite covering by closed balls.
Proofs of some needed auxiliary facts are contained in the Appendix (Section~\ref{section:appendix}).

\medskip

\section{Preliminaries and some general facts}\label{section:preliminaries}

Throughout the paper, $\N$ denotes the set of strictly positive integers, while $\N_0:=\N\cup\{0\}$ is the set of nonnegative integers. Given a set $\Gamma$ and $n\in\N_0$, by $[\Gamma]^n$ we mean the set of all $n$-element subsets of $\Gamma$, and by $[\Gamma]^{<\infty}$  the set of all finite subsets of $\Gamma$.
Thus $[\Gamma]^{<\infty}=\bigcup_{n\ge0}[\Gamma]^n$.

We consider only nontrivial real normed spaces. If $X$ is a normed space then $X^*$ is its dual Banach space, and $B_X$ and $S_X$ are the closed unit ball and the unit sphere of $X$. Moreover, we denote by $B(x,\epsilon)$ and $U(x,\epsilon)$ the closed and the open ball with radius $\epsilon$ and center $x$, respectively.  By a {\em ball} in $X$ we mean a closed or open ball of positive radius in $X$. If $B\subset X$ is a ball then $c(B)$ and $r(B)$ denote its center and radius, respectively. Other notation is standard, and various topological notions refer to the norm topology of $X$, if not specified otherwise. A set $B\subset X$ will be called a {\em body} if it is closed, convex and has nonempty
interior. For $x,y\in X$, $[x,y]$ denotes the closed segment in $X$ with
endpoints $x$ and $y$, and $(x,y)=[x,y]\setminus\{x,y\}$ is the
corresponding ``open'' segment.

Let $\mathcal{F}$ be a family of nonempty sets in a normed space $X$. By $\bigcup \mathcal F$ we mean the union of all members of $\mathcal F$.
A point $x\in X$ is a {\em regular point} for $\mathcal F$ if it has a neighborhood that meets at most finitely many members of $\mathcal F$. 
Points that are not regular are called {\em singular}. 
Notice that the set of regular points is an open set.
For any $x\in X$ we denote
$$
\mathcal{F}(x):=\{F\in\mathcal{F}: x\in F\}.
$$
Thus $\mathcal F$ is a covering of $X$ if and only if $\mathcal{F}(x)\ne\emptyset$ for each $x\in X$. 

\begin{definition}\label{D: finiteness} The family $\mathcal F$ is called:
\begin{enumerate}[(a)]
\item {\em star-finite} if each of its members intersects only finitely many other members of $\mathcal F$ (cf. \cite[p. 317]{Engel}); \item {\em simple} if its members are pairwise disjoint;
\item {\em point-finite} ({\em point-countable}) [{\em $n$-finite} ($n\in\N$)] if each $x\in X$ is contained in at most finitely many (countably many) [$n$] members of $\mathcal F$;
\item {\em locally finite} if each $x\in X$ is a regular point for $\mathcal F$.   
\end{enumerate}
\end{definition}

It is evident that simple families are star-finite, and star-finite families are point-finite (and hence point-countable).

A {\em minimal covering} is a covering whose no proper subfamily is a covering. Notice that a covering need not contain any minimal subcovering (consider e.g.\ the covering consisting of $nB_X$, $n\in\N$). However, it is easy to see that the intersection of a chain of point-finite coverings is again a covering. Thus by Zorn's lemma {\em every point-finite (hence every star-finite) covering contains a minimal subcovering.}

\medskip

\subsection{Cardinality properties} 
Next results describe relations between the cardinality of certain coverings of a topological space $T$ and its density character $\mathrm{dens}(T)$ (i.e., the smallest cardinality of a dense subset of $T$).

\begin{lemma}
Let $T$ be an infinite Hausdorff topological space, and $\mathcal{F}$ a point-countable family of nonempty open subsets of $T$. Then $|\mathcal{F}|\le\mathrm{dens}(T)$.
\end{lemma}

\begin{proof}
Fix a dense (necessarily infinite) set $D\subset T$. For each $A\in\mathcal{F}$ choose some $f(A)\in D\cap A$, obtaining in this way a function $f\colon\mathcal{F}\to D$ such that the subfamilies $f^{-1}(d)\subset\mathcal F$, $d\in D$, are all at most countable. It is clear that these subfamilies are pairwise disjoint. Now we obtain
$$
|\mathcal{F}|=|\bigcup_{d\in D}f^{-1}(d)\,|\le |D\times\N|=|D|,
$$
which completes the proof by arbitrariness of the dense set $D$.
\end{proof}

\begin{observation}\label{O: card upper bound}
The above cardinality estimate applies whenever $\mathcal{F}$ is a point-finite family of sets with nonempty interior (and $T$ as above). Indeed, it suffices to consider the family 
$\mathcal{F}'=\{\mathrm{int}\,F: F\in\mathcal{F}\}$.
\end{observation}

Since we are interested in star-finite coverings of normed spaces by bodies, we will always have that the cardinality of such a covering is not greater than the density character of the space.

\begin{lemma}
Let $X$ be an infinite-dimensional normed space, $r>0$, and $\mathcal{B}$ a covering of $X$ by balls
of radius at most $r$. Then $|\mathcal{B}|\ge \mathrm{dens}(X)$.
\end{lemma}

\begin{proof}
Let $E\subset X$ be a maximal $3r$-dispersed set, that is: $\|y-z\|\ge 3r$ for any distinct $y,z\in E$,
and for each $x\in X$ there is $y\in E$ such that $\|x-y\|<3r$. Then the set $D:=\bigcup_{n\in\N}(1/n)E$
is dense and $|E|=|D|\ge\mathrm{dens}(X)$. Since 
$E\subset\bigcup\mathcal{B}$ and
each member of $\mathcal{B}$
contains at most one element of $E$, we conclude that $|\mathcal{B}|\ge|E|\ge\mathrm{dens}(X)$.
\end{proof}

Let $X$ be a normed space. Recall that a set $A\subset X^*$ is {\em total} if
$^\perp A:=\bigcap_{x^*\in A}\mathrm{Ker}(x^*)=\{0\}$. Thus if $A$ is total then
$\overline{\mathrm{span}}^{w^*}\,A=(^\perp A)^\perp=X^*$. It follows that if $A$ is total and infinite then
$$
w^*\text{-}\mathrm{dens}(X^*):=\mathrm{dens}(X^*,w^*)\le|\mathrm{span}_\Q\,A|=|A|\,,
$$
where $\mathrm{span}_\Q\,A$ is the ``rational span'' of $A$.

\begin{proposition}\label{p: wstardensityCardinality}
Let $X$ be an infinite-dimensional normed space. Suppose that $X$ admits a covering $\B$ by closed bounded convex sets  such that some $x_0\in X$ belongs to only finitely many elements of $\B$. Then $w^*\text{\rm-dens}(X^*)\leq|\B|$.
\end{proposition}

\begin{proof}
By translation we may (and do) assume that $x_0=0$. 
Define $\B':=\B\setminus\B(0)$. Since $\bigcup\B(0)$ is bounded, by homogeneity we can (and do) assume that $S_X\subset \bigcup\B'$.
Set $\B^{''}\coloneqq\{B\in \B': B\cap S_X\neq\emptyset\}$. By the Hahn-Banach theorem, for each $B\in\B''$ there exists $x^*_B\in S_{X^*}$ such that $0=x^*_B(0)<\inf x^*_B(B)$. Since $S_X\subset \bigcup \B''$,
the family $\{x^*_B\}_{B\in\B''}$ is total and hence infinite. Consequently,
$w^*\text{\rm-dens}(X^*)\leq|\B''|\le|\B|$.
\end{proof}

From the previous result we deduce the exact size of a point-finite (star-finite) covering for a wide class of Banach spaces, more precisely the class of weakly Lindel\"of determined Banach spaces (WLD). The class of WLD Banach spaces, that generalizes the class of WCG Banach spaces, has been studied first in \cite{ArgMer} (see also \cite{Hajeketal} for more details).

\begin{corollary}
Let $X$ be a WLD Banach space. Suppose that $\B$ is a point-finite covering by bodies of $X$. Then $\dens(X)=|\B|$.
\end{corollary}

\begin{proof}
By Observation~\ref{O: card upper bound} we have $|\B|\leq \dens(X)$. The other inequality follows combining Proposition \ref{p: wstardensityCardinality} with the fact that
$\mathrm{dens}(X)=w^*\text{\rm-dens}(X^*)$ (see \cite[Proposition 5.40]{Hajeketal}).
\end{proof}

\smallskip

\subsection{Structure properties} Let us state some simple properties of star-finite coverings, which will be used in the sequel.

\begin{observation}\label{O: basic}
Let $\mathcal F$ be a star-finite covering by closed sets of a normed space $X$. Then it 
has the following properties.
\begin{enumerate}[(a)]
\item The set $D:=\bigcup_{F\in\mathcal{F}}\partial F$ is closed.
\item A point $x\in X$ is regular for $\mathcal{F}$ if and only if
$x\in\mathrm{int}\,[\bigcup \mathcal{F}(x)]$.
\item If $x$ is a singular point of $\mathcal F$ then $x\in\bigcap_{F\in\mathcal{F}(x)}\partial F$.
\item If $\mathcal F$ is countable then $H:=\{x\in D: |\mathcal{F}(x)|=1\}$ is a $G_\delta$ set.
\end{enumerate}
\end{observation}

\begin{proof}
(a) If $x\notin D$ then $x\in U:=\bigcap_{F\in\mathcal{F}(x)}\mathrm{int}\,F$. Since
$\mathcal{F}(x)$ contains only finitely many sets each of which intersects only finitely many members of $\mathcal{F}\setminus\mathcal{F}(x)$, it follows that 
the set $U\setminus \bigcup[\mathcal{F}\setminus\mathcal{F}(x)]$ is an open neighborhood of $x$ which is disjoint from $D$.
This proves that $X\setminus D$ is open.

(b) The implication ``$\Leftarrow$'' follows in a similar way to (a), now starting from the set
$U:=\mathrm{int}\,[\bigcup \mathcal{F}(x)]$. To show the other implication, assume that $x$ is a regular point for $\mathcal F$, that is, there exists an open neighborhood $V$ of $x$ for which the subfamily
$\{F\in\mathcal{F}: F\cap V\ne\emptyset\}$ is finite. Now star-finiteness of $\mathcal F$ easily implies that there exists a neighborhood $U\subset V$ of $x$ such that $U$ is contained in $\bigcup\mathcal{F}(x)$.

(c) If $x$ is singular then $x\notin\mathrm{int}[\bigcup\mathcal{F}(x)]$ by (b), and hence
$x\notin\bigcup_{F\in\mathcal{F}(x)}(\mathrm{int}\,F)$. Thus $x\in\bigcap_{F\in\mathcal{F}(x)}\partial F$.

(d) Write $\mathcal{F}=\{F_n\}_{n\in\N}$. Then 
$D\setminus H=D\cap\bigcup_{m\ne n}(F_m\cap F_n)$ is an $F_\sigma$ set in $D$, hence $H$ is $G_\delta$ in $D$. Since $D$ is $G_\delta$ in $X$, it follows that $H$ is $G_\delta$ in $X$.
\end{proof}

\begin{lemma}\label{L: one ball}
Let $C_1,\dots, C_n$ and $B$ be closed convex sets in an infinite-dimen\-sional normed space $X$. If
$B$ is bounded and $\{C_i\}_1^n$
does not cover $B$ then $\partial B\setminus\bigcup_{i=1}^n C_i$ is weakly dense in
$B\setminus\bigcup_{i=1}^n C_i$. In particular, $\{C_i\}_1^n$
does not cover $\partial B$.
\end{lemma}

\begin{proof}
Let $B$ have interior points (otherwise there is nothing to prove).
Proceeding by contradiction, assume there exists $x\in B\setminus\bigcup_{i=1}^n C_i$
which does not belong to $\overline{\partial B\setminus\bigcup_{i=1}^n C_i}^{w}$. We have
$$\textstyle
\partial B\subset \overline{\partial B\setminus\bigcup_{i=1}^n C_i}^{w}\,\cup
\bigcup_{i=1}^n C_i\,=:E,
$$
where $E$ is a weakly closed set that does not contain $x$. Let $W$ be a weak neighborhood of $x$ which is disjoint from $E$. But then $W\cap \partial B=\emptyset$, which is impossible since $W$ contains a line.
This contradiction completes the proof.
\end{proof}

\begin{corollary}\label{C: one ball}
Let $\mathcal{B}$ be a minimal star-finite covering by bounded closed convex sets of an infinite-dimensional normed space $X$.
Then the boundary of each $B\in\mathcal{B}$ contains a nonempty relatively open set which does not meet other members of $\mathcal B$.
\end{corollary}

\begin{proof}
Given $B$, let $C_1,\dots, C_n$ be the members of $\mathcal{B} \setminus\{B\}$ that intersect $B$.
By minimality, $\{C_i\}_1^n$ does not cover $B$. By Lemma~\ref{L: one ball}, 
$\partial B\setminus\bigcup_{i=1}^n C_i\ne\emptyset$.
\end{proof}

\smallskip

\subsection{Covering normed spaces of countable dimension} In the rest of this section we will show that each normed space with countable dimension can be covered by a star-finite family of closed balls. This result is achieved by covering inductively a nested sequence of finite-dimensional subspaces.\\
Let $A$ be a set in a metric space $(X,d)$, and let $\delta>0$. Recall that a set $E\subset X$ is a
{\em $\delta$-net} for $A$ if $\mathrm{dist}(a,E)<\delta$ for each $a\in A$. 

In what follows, we shall use several times the following simple fact.

\begin{observation}\label{O: ball reduction}
Let $Z$ be a convex subset of  a normed space $X$. Let $B_1$ and $B_2$ be two closed balls in $X$, such that $c(B_1),c(B_2)\in Z$, then $B_1\cap B_2=\emptyset$ if and only if $Z\cap B_1\cap B_2=\emptyset$.
\end{observation}

\begin{proof}
The proof is done observing that two balls intersect if and only if the distance of their centers is not greater than the sum of their radii if and only if the balls intersect in the segment connecting the centers.
\end{proof}

The key step in the proof of Theorem \ref{T: star-finite countable dimension} is the next lemma, which proves that each open subset $A$ of a finite-dimensional normed space admits a star-finite covering by closed balls whose singular points accumulate on the boundary of $A$. 

\begin{lemma}\label{L: totallybounded}
Let $X$ be a normed space, and $Y\subset X$ a finite-dimensional subspace. Let $C\subset Y$ be  a closed set such that $Y\setminus C\ne\emptyset$. Then there exists a star-finite family $\B$ of closed balls of $X$ such that:
\begin{enumerate}[(a)]
\item  $c(B)\in Y$ and $B\cap C=\emptyset$ for each $B\in\B$;
\item $Y\setminus C\subset\bigcup\B$;
\item the singular points of $\mathcal B$ are contained in $C$.
\end{enumerate}  
\end{lemma}

\begin{proof}
Let us define
\begin{align*}
A_{h,k}&\textstyle :=\{y\in Y\setminus C: \frac1{k+1}<\mathrm{dist}(y,C)\leq\frac1k,h\leq\|y\|< h+1\}
\ \ (h\in\N_0, k\in\N),\\
A_{h,0}&:=\{y\in Y\setminus C:  {1}<\mathrm{dist}(y,C),h\leq\|y\|< h+1\}
\ \ (h\in\N_0),
\end{align*}
where for $C=\emptyset$ we put $\mathrm{dist}(y,C):=\infty$.
For each $h,k\in\N_0$, the bounded set $A_{h,k}\subset Y$ admits a finite $\frac1{2(k+1)}$-net
$E_{h,k}\subset A_{h,k}$. Consider the family
$$\textstyle
\B:=\left\{z+\frac1{2(k+1)}B_X:\;z\in E_{h,k},\;k,h\in\N_0\right\}
$$
which clearly satisfies (a). Since $Y\setminus C\subset\bigcup_{h,k\in\N_0}A_{h,k}\,$, 
the condition (b) easily follows by the choice of the sets $E_{h,k}$. 

Now let us show (c). Let $x\in X$ be a singular point of $\B$. Then $\B$ contains a sequence $\{B_n\}$
of pairwise distinct closed balls such that $\mathrm{dist}(x,B_n)\to0$. For each $n\in\N$ there are $h_n,k_n\in\N_0$
such that
$$\textstyle
c(B_n)\in E_{h_n,k_n}\quad\text{and}\quad
r(B_n)=\frac1{2(k_n+1)}\le\frac12\,.
$$
It is easy to see that $\{h_n\}$ is necessarily bounded and $\{k_n\}$ is unbounded.
So we can (and do) assume that $k_n\to\infty$. But then we obtain
\begin{align*}
\mathrm{dist}(x,C)&\le\|x-c(B_n)\|+\mathrm{dist}(c(B_n),C) \\
&\textstyle\le \mathrm{dist}(x,B_n) + r(B_n)+\frac1{k_n}\to0\,,
\end{align*}
and hence $x\in C$.

Finally, proceeding by contradiction, let us show that $\B$ is star-finite. So assume that $\B$ is not star-finite. 
There exists an infinite subfamily $\{B_n\}_{n\in\N}\subset \B$ such that $B_1\cap B_n\neq \emptyset$ for each $n\ge2$. 
By Observation~\ref{O: ball reduction}, $B_1\cap B_n\cap Y\neq \emptyset$, $n\ge2$.
Fix arbitrarily $y_n\in B_1\cap B_n\cap Y$. Since $B_1\cap Y$ is compact, there exists a subsequence $\{y_{n_k}\}$ that converges to some $y\in B_1$. But then $y$ is a singular point of $\B$ 
which, by (a), does not belong to $C$. This contradicts (c), and we are done.
\end{proof}

Finally let us prove the main result of the present section.

\begin{theorem}\label{T: star-finite countable dimension}
Let $X$ be a normed space such that $\dim X=\aleph_0$. Then $X$ has a star-finite covering $\B$ by closed balls.
\end{theorem}

\begin{proof}
Let $\{e_n\}_{n\in\N}$ be a Hamel basis of $X$. We set 
$Y_0:=\{0\}$, and  $Y_n\coloneqq\mathrm{span}\{e_1,...,e_n\}$ for $n\in\N$. 
We will inductively define families $\B_n$ ($n\in\N_0$) of closed balls, satisfying for each $n\in\N_0$ the following conditions:
\begin{enumerate}
\item[($\mathrm P^1_n$)] $\B_n$ is star-finite;
\item[($\mathrm P^2_n$)] $Y_n\subset C^n:=\bigcup(\B_0\cup\dots\cup \B_n)$;
\item[($\mathrm P^3_n$)] $C^{n}$ is closed;
\item[($\mathrm P^4_n$)] $\bigcup\B_n$ is disjoint from $\bigcup(\bigcup_{k<n}\B_k)$.
\end{enumerate}

\smallskip\noindent
To start, put $\B_0:=\{B_X\}$ and notice that the conditions ($\mathrm P^1_0$)-($\mathrm P^4_0$) are trivially satisfied. 
Now, take $n\in\N$ and assume  we have already defined $\B_k$ for $k\le n-1$. 
Since $C:=C^{n-1}\cap Y_n$ is closed,
by Lemma~\ref{L: totallybounded} there exists a star-finite family $\B_n$ of closed balls of $X$,
all centered in $Y_n$,  such that $Y_n\cap C^{n-1}\cap\bigcup\B_n=C\cap\bigcup\B_n=\emptyset$, 
$Y_n\setminus C\subset \bigcup \B_n$, and all singular points of 
$\B_n$ belong to $C$. Since both $C^{n-1}$ and $\bigcup \B_n$ are unions of closed balls centered in $Y_n$, we can apply Observation~\ref{O: ball reduction} to obtain that
$C^{n-1}\cap\bigcup\B_n=\emptyset$, which shows ($\mathrm P^4_n$). Moreover,
$Y_n=C\cup(Y_n\setminus C)\subset C^{n-1}\cup\bigcup\B_n$, which is ($\mathrm P^2_n$). It remains to verify ($\mathrm P^3_n$). To this end, consider $x\in\overline{\bigcup\B_n}\,\setminus\bigcup\B_n\,$
and notice that $x$ is a singular point of $\B_n$, which implies that $x\in C\subset C^{n-1}$. Consequently,
$\overline{C^n}= C^{n-1}\cup\overline{\bigcup\B_n}\,\subset C^{n-1}\cup\bigcup\B_n=C^n$ which means that $C^n$ is closed.\\
Finally, let $\B=\bigcup_{n\in\N_0}\B_n$. By property ($\mathrm P^2_n$), we easily get that $\B$ is a covering. Since the sets $\bigcup \B_n$ ($n\in\N_0$) are pairwise disjoint, we immediately obtain  star-finiteness of $\B$. The proof is complete.
\end{proof}

\medskip

\section{Prohibitive conditions for coverings by closed balls}\label{section:prohibitive}

In the present section, we provide results on non-existence of star-finite or simple coverings of some Banach spaces. Main of these results are contained in Corollary~\ref{C: suffcond}, Corollary~\ref{C: no countable}, Theorem~\ref{T: c_0dirbuona}, and Corollary~\ref{C: no simple c0Gamma}.

\subsection{Rotundity and differentiability conditions} 

\begin{definition}\label{D: propertyI}
	Let $X$ be a normed space, and $\alpha$ a cardinal. 
	\begin{enumerate} 
		\item Given $\epsilon>0$, we say that a point $x\in S_X$ has property $(\mathcal I_{\alpha,\epsilon})$ if, whenever $\B$ is a family of pairwise disjoint closed balls of radius 1 not intersecting $B_X$ and such that $|\B|=\alpha$, we have 
		$$\textstyle \sup_{B\in\B}\mathrm{dist}(x,B)>\epsilon.$$
		
		\item We say that $X$ has property $(\mathcal I_{\alpha})$ if,
		for each $x\in S_X$ there exists $\epsilon>0$ such that $x$ has property $(\mathcal I_{\alpha,\epsilon})$.
		\item We say that $X$ has property $(U\mathcal I_{\alpha})$ if there exists $\epsilon>0$ such that each $x\in S_X$ has property $(\mathcal I_{\alpha,\epsilon})$.
		\item We denote
		$$
		\K(X,\alpha):=\sup\bigl\{\mathrm{sep}\,A:\, A\subset S_X,\ |A|=\alpha\bigr\},
		$$
		where\ \ $\mathrm{sep}\,A:=\inf\{\|a-b\|:\, a,b\in A,\ a\neq b\}$.
	\end{enumerate}
\end{definition}

\begin{remark}\label{R:I}
	Let $\alpha,\beta$ be cardinals such that $\alpha<\beta$, $x\in S_X$, and $\epsilon>0$.
	\begin{enumerate}[(a)]
		\item If $x$ has property $(\mathcal I_{\alpha,\epsilon})$ then $x$ has property $(\mathcal I_{\beta,\epsilon})$.
		\item  If $X$ has property $(U\mathcal I_{\alpha})$ then $X$ has property $(\mathcal I_{\alpha})$.
		\item It is clear that if $B$ is a closed ball in $X$ and $u\in \partial B$, then for each $r\in(0,r(B))$ there exists a closed ball $B'\subset B$ such that $r(B')=r$ and $u\in\partial B'$. This simple observation easily implies that:
		{the point $x$ has property $(\mathcal{I}_{\alpha,\epsilon})$ if and only if,
				whenever $\B$ is a family of pairwise disjoint closed balls not intersecting $B_X$ such that $|\B|=\alpha$
				and $\inf_{B\in\mathcal{B}}r(B)\ge\rho>0$, we have\ \  
						$\textstyle \sup_{B\in\B}\mathrm{dist}(x,B)>\rho\epsilon.$}
		\item Notice also that if $\alpha$ is an infinite cardinal then: $X$ has property $(U\mathcal{I}_\alpha)$ if and only if
		there exists $\epsilon>0$ such that if $\mathcal{B}$ is a disjoint family of closed balls of radius $1$ with $|\mathcal{B}|=\alpha$, and $x_B\in \partial B$ ($B\in\mathcal{B}$), then\ \  $\mathrm{diam}\{x_B\}_{B\in\mathcal{B}}>\epsilon$.
	    \item We clearly always have 
		$ \K(X,\alpha)\le 2$. Moreover, $\K(X,{\aleph_0})$ coincides with $\K(X)$, the {\em Kottman's (separation) constant} of a Banach space $X$; see \cite{Kott}. 		
	\end{enumerate}  
\end{remark}

\noindent The next lemma provides a characterization of property $(U\mathcal {I}_\alpha)$ in terms of $K(X,\alpha)$.

\begin{lemma}\label{L: costante di Kottman}
	Let $X$ be an infinite-dimensional normed space and let $\alpha$ be an infinite cardinal. Then $\K(X,\alpha)<2$ if and only if $X$ has $(U\mathcal {I}_\alpha)$.
\end{lemma}

\begin{proof}
	First assume that $\K(X,\alpha)=2$, and fix an arbitrary $\epsilon\in(0,2)$. There exists a set $A\subset S_X$
	with $\mathrm{sep}\, A>2-\epsilon$ and $|A|=\alpha$. Then the balls $B_a:=B(a,1-\epsilon/2)$, $a\in A$, are pairwise disjoint, and moreover $y_a:=(\epsilon/2)a\in B_a$. Clearly, $\mathrm{diam}\{y_a\}_{a\in A}\le\epsilon$. By multiplying everything by $r:=(1-\epsilon/2)^{-1}$
	we obtain pairwise disjoint balls $r B_a$ ($a\in A$) of radius $1$, and points $z_a:= r y_a\in r B_a$ such that
	$\mathrm{diam}\{z_a\}_{a\in A}\le r\epsilon=2\epsilon/(2-\epsilon)$. Since $\epsilon$ can be arbitrarily small, $X$ fails $(U\mathcal I_\alpha)$ by Remark~\ref{R:I}(d).
	
	Now, assume that $X$ fails $(U\mathcal I_\alpha)$, and fix an arbitrary $\epsilon\in(0,1)$.
	By Remark~\ref{R:I}(d), there exist pairwise disjoint balls $B_\gamma:=B(c_\gamma,1)$ ($\gamma<\alpha$) and points $y_\gamma\in B_\gamma$ with
	$\mathrm{diam}\{y_\gamma\}_{\gamma<\alpha}\le\epsilon/2$. By translation, we can (and do) assume that
	$\{y_\gamma\}_{\gamma<\alpha}\subset \epsilon B_X$. Since the origin belongs to at most one of the balls $B_\gamma$,
	by excluding such a ball we can (and do) assume that $0\notin B_\gamma$ ($\gamma<\alpha$).
	Then $1<\|c_\gamma\|\le\|c_\gamma-y_\gamma\|+\|y_\gamma\|\le 1+\epsilon$ for each $\gamma<\alpha$, and
	$\|c_\gamma-c_\beta\|>2$ whenever $\gamma \ne \beta$. Consider the set $A$ of all the points $x_\gamma:=c_\gamma/\|c_\gamma\|$ ($\gamma<\alpha$).
	Then $\|x_\gamma-c_\gamma\|=\|c_\gamma\|-1\le\epsilon$ and hence
	for $\gamma\ne \beta$ we have $\norm{x_\gamma-x_\beta}\ge\norm{c_\gamma-c_\beta}-\norm{x_\gamma-c_\gamma}-\norm{x_\beta-c_\beta}
	>2-2\epsilon$. Since $\mathrm{sep}\, A\ge 2-2\epsilon$ and $\epsilon$ can be arbitrarily small, we conclude that $\K(X,\alpha)=2$.
\end{proof}

Next theorem shows that Banach spaces satisfying condition $(\mathcal I_{{\aleph_0}})$ do not admit any star-finite covering by closed balls. In order to prove this result we need a simple lemma.

\begin{lemma}\label{L: uncountable star.finite} Let $X$ be a normed space, and $Y$ its separable subspace. Suppose that $\B$ is a star-finite covering of $X$ by closed balls such that uncountably many elements of $\B$ intersect $Y$. Then $X$ fails property $(\mathcal I_{\aleph_1})$.
\end{lemma}

\begin{proof}
	Let us consider the uncountable family $\B':=\{B\in\B: B\cap Y\neq\emptyset\}$ and, 
	for each $C\in\B'$, let us consider $y_C\in Y\cap C$. 
	By Zorn's lemma, there exists a maximal simple subfamily $\mathcal{C}'$ of $\mathcal{B}'$. Notice that, since the family  $\mathcal{B}'$ is uncountable and star-finite, $\mathcal{C}'$ must be uncountable.   
	If we denote $\mathcal C'_m:=\{C\in\mathcal C' : r(C)\geq \frac1m\}$
	($m\in\N$), it is clear that there exists $n\in\N$ such that $\mathcal C'_n$ is uncountable.  
	Since $Y$ is separable, there exists a condensation point $\overline y\in Y$ for the set 
	$U:=\{y_C : C\in\mathcal C'_n\}$. Moreover, there exists $\widetilde B\in \B'$ such that $\overline y\in \widetilde B$; since $\B'$ is star-finite, we have $\overline y\in\partial \widetilde B$, moreover, only finitely many elements of $\mathcal C'_n$ intersect $\widetilde B$. It easily follows that $X$ fails property $(\mathcal I_{\aleph_1})$.
\end{proof}

\begin{theorem}\label{T: trebolle}
	Let $X$ be an infinite-dimensional Banach space satisfying property $(\mathcal I_{{\aleph_0}})$. Then $X$ does not admit star-finite coverings by closed balls.	
\end{theorem}

\begin{proof}
	Proceeding by contradiction, assume that such a covering $\B$ exists. Let us consider $Y$, a separable infinite-dimensional  subspace of $X$.
	By Lemma~\ref{L: uncountable star.finite} and since $X$ has property $(\mathcal I_{{\aleph_0}})$ (and hence property $(\mathcal I_{\aleph_1})$), the family $\B':=\{B\cap Y : B\in\B, B\cap Y\neq\emptyset\}$ must be
	countable. Moreover, we can (and do) assume that $\mathcal B'$ is a minimal covering of $Y$, and denote 
	$$\textstyle
	D:=\bigcup_{B\in\B'} \partial B\,,\quad
	H:=\{x\in D: |\mathcal{B'}(x)|=1\}.
	$$
	Observe that since $Y$ is infinite-dimensional and $\B'$ is minimal, $H$ is nonempty by Corollary~\ref{C: one ball}.
	By Observation~\ref{O: basic}(d), $H=\bigcup_{B\in\B'}(\partial B\cap H)$ is a Baire space.
	Therefore there exists $B_0\in\B'$ such that $\partial B_0\cap H$ is not nowhere dense in $H$.
	Using the fact that $\partial B_0\cap H$ is a relatively open set in $\partial B_0$ 
	(see Corollary~\ref{C: one ball}), it easily follows that there exist $x_0\in\partial B_0$ and $\epsilon>0$ so that
	\begin{equation}\label{E: UH1}
	U(x_0,\epsilon)\cap H\subset \partial B_0\cap H.
	\end{equation}
	Clearly, $x_0$ is a singular point for $\B'$. Since $\B'$ is star-finite, there exists a sequence $\{y_n\}\subset Y$ such that $y_n\to x_0$, $y_n\in C_n\in\B'$ and the sets $C_n$ ($n\in\N$) are pairwise distinct. Now, for each $n\in\N$, there exists $B_n\in\B$ such that $C_n=B_n\cap Y$. Let  $ r(B_n)$ be the radii of the balls $B_n$ ($n\in\N$) and consider the following two cases.
	\begin{enumerate}
		\item $ r(B_n)\not\to0$. Let $D_0\in \B$ be such that $B_0=D_0\cap Y$. By considering a suitable subsequence we can suppose without any loss of generality that: (a) there exists $\alpha>0$ such that $ r(B_n)>\alpha$, whenever $n\in\N$, and such that $ r(D_0)>\alpha$; (b) the sets $B_n$ ($n\in\N$) and $D_0$ are pairwise disjoint. 
		\item $ r(B_n)\to0$. Since $Y$ is infinite-dimensional and $\B'$ is minimal,  by Corollary~\ref{C: one ball}, for each $n\in\N$ there exists $z_n\in H\cap C_n$. In particular, $z_n\to x_0$ and hence, since $(x_0+\epsilon B_Y)\cap H\subset\partial B_0$, we have that eventually $z_n\in\partial B_0$. Hence, eventually $C_n\cap B_0\neq\emptyset$. 
	\end{enumerate}	  
	We have a contradiction, in the first case  since $X$ has property $(\mathcal I_{\aleph_0})$,  and in the latter case since $\B'$ is star-finite. This concludes the proof.
\end{proof}

The rest of the present subsection is devoted to finding sufficient conditions for a Banach space to satisfy property $(\mathcal I_{\aleph_0})$. For this purpose let us recall the following definition from \cite{DEVETIL}.

\begin{definition}[see {\cite[Definition~4.6]{DEVETIL}}]
	We shall say that $x\in S_X$ is a {\em locally non-D2} (or {\em LND2}) {\em point} of $B_X$ if there exists $\delta>0$ such that
	\begin{equation*}\label{nd2}
	\textstyle
	\text{diam}\bigl\{y\in S_X:\;\|\frac{x+y}2\|\geq1-\delta\bigr\}<2\,.
	\end{equation*}
\end{definition}
\medskip  

\noindent The following lemma immediately follows by \cite[Lemma~4.5]{DEVETIL}.

\begin{lemma}\label{L: trebollelur} 
	Let $X$ be a normed space, $\epsilon\ge0$, and
	$B_0,B_1,B_2\subset X$ three closed balls of radius one whose interiors are pairwise disjoint.
	Consider three points $y_i\in\partial B_i$, $i=0,1,2$,
	and denote $x_0=y_0-d_0$ where $d_0$ is the center of $B_0$.
	If $\mathrm{diam}\{y_0,y_1,y_2\}\leq\epsilon\,$ then
	\begin{equation}\label{E: 3b}
	\textstyle
	\mathrm{diam}\bigl\{y\in S_X:\|x_0+y\|\geq2-\epsilon\bigr\}\geq2-2\epsilon\,.
	\end{equation}
\end{lemma}

\medskip

\noindent For $f\in S_{X^*}$ and $\alpha\in[0,1)$, we consider the closed convex cone
$$
\C(\alpha, f)=\{x\in X: f(x)\ge\alpha\|x\|\}.
$$
The following observation is an analogue of \cite[Observation 2.1]{DEVETIL} for uniformly Fr\'echet smooth norms.

\begin{observation}\label{O: Fr}
	Suppose that $X$ is a Banach space with uniformly Fr\'echet smooth norm. Then
	for each $\alpha\in(0,1)$ there exists $\epsilon>0$ such that for each $x\in S_X$ there exists  $f_x\in S_{X^*}$ with the following property: 
	\begin{equation}\label{E: fr}
	[x-\C(\alpha,f_x)]\cap[x+\epsilon B_X]\subset B_X\,.
	\end{equation}
	\rm
\end{observation}

\begin{proof} 
	For each $x\in S_X$, let $f_x\in S_{X^*}$ be the Fr\'echet derivative of $\|\cdot\|$ at $x$. Since the norm of $X$ is uniformly Fr\'echet smooth, for each $\alpha\in(0,1)$ there exists $\epsilon>0$ such that, for each $x\in S_X$, we have
	$\bigl|\|x+h\|-1-f_x(h)\bigr|\leq\alpha\|h\|$, whenever $h\in\epsilon B_X$.
	Thus, for $h\in [-\C(\alpha,f_x)]\cap\epsilon B_X$, we obtain
	$\|x+h\|\leq 1+f_x(h)+\alpha\|h\|\le 1$,
	and hence $x+h\in B_X$. This completes the proof.
\end{proof}

\begin{definition}[{see \cite[Definition~2.2]{DEVETIL}}]\label{D: cs} Let  $x\in S_X$ and $\epsilon>0$.
	We say that $x$ is an $\epsilon$-\emph\cs\  point of $B_X$
	if there exists $f_x\in \SX$ such that
	$$\textstyle[x-\C(\frac17,f_x)]\cap[x+\epsilon B_X]\subset B_X\,,$$
	that is, \eqref{E: fr} holds for $\alpha=1/7$.
\end{definition}

Observe that, if the norm of $X$ is uniformly Fr\'echet smooth, then, by Observation~\ref{O: Fr}, there exists $\epsilon>0$ such that each $x\in S_X$ is an $\epsilon$-\cs\ point of $B_X$. 

\begin{proposition}\label{P: suffcond}
	Let $X$ be a Banach space and $x\in S_X$. Let us consider the following conditions:
	\begin{enumerate}[(i)]
		\item $X$ is uniformly Fr\'echet smooth;
		\item there exists $\epsilon>0$ such that the set of all $\epsilon$-cone smooth points of $B_X$ is dense in $S_X$;
		\item $\K(X)\equiv \K(X,\aleph_0)$, the Kottman's constant of $X$, satisfies $\K(X)<2$;
		\item $x$ is an LUR point;
		\item $x$ is an LND2 point;
		\item $x$ is a Fr\'echet smooth and strongly exposed point of $B_X$;
		\item $x$ is a Fr\'echet smooth point and the unique norm-one functional $f_x\in X^*$ that supports $B_X$ at $x$ determines a slice $\Sigma$ of $B_X$ such that
		$\mathrm{diam}(\Sigma)<2$.
	\end{enumerate}
Then the following implications hold.
\begin{enumerate}[(a)]
\item If $(i)$ or $(ii)$ is satisfied then $X$ has property $(U\mathcal I_2)$.
\item If (iii) is satisfied then $X$ has property $(U\mathcal I_{\aleph_0})$.
\item If at least one of the conditions $(iv)$--$(vii)$ is satisfied then the point $x$ has property 
$(\mathcal I_{2,\epsilon})$ for some $\epsilon>0$.
\end{enumerate}
\end{proposition}

\begin{proof}  {\em(a)} By the observation immediately after Definition~\ref{D: cs}, (i) implies (ii). Moreover, if (ii) is satisfied, \cite[Lemma~4.1]{DEVETIL} easily implies that $X$ has property $(U\mathcal I_2)$.

\smallskip\noindent
{\em(b)}	It follows immediately by Lemma~\ref{L: costante di Kottman}.

\smallskip\noindent
{\em (c)}
	It is clear that (iv) implies (v). Moreover, if (v) is satisfied it follows by Lemma~\ref{L: trebollelur} that $x$ has property $(\mathcal I_{2,\epsilon})$ for some $\epsilon>0$.  
	Finally, it is clear that (vi) implies (vii). Let us prove that if (vii) is satisfied then $x$ has property $(\mathcal I_{2,\epsilon})$ for some $\epsilon>0$. We proceed as in the last part of the proof of \cite[Theorem~4.9]{DEVETIL}. Suppose on the contrary that, for each $\epsilon>0$, $x$ fails property $(\mathcal I_{2,\epsilon})$. Then there exist sequences $\{w_n\}$, $\{u_n\}$ in $X$ such that
	\begin{itemize}
		\item for each $n\in\N$, there exist  $B_n,C_n$, closed balls of radius 1, such that $B_X,B_n,C_n$ are pairwise disjoint and $w_n\in \partial B_n,u_n\in \partial C_n$;
		\item $\mathrm{diam}\{x,w_n,u_n\}\to0$.
	\end{itemize}
	By Lemma~\ref{L: trebollelur}, for each $\delta>0$, we have that
	$\textstyle\mathrm{diam}\{y\in S_X:\|\frac{x+y}2\|\geq1-\delta\}=2$.
	This easily implies existence of a sequence $\{y_n\}\subset S_X$ such that
	$\|\frac{x+y_n}2\|\to1$, and $\mathrm{diam}(\{y_n\}_{n\geq n_0})=2$ for each $n_0\in\N$.
	By convexity of the norm,
	for each $n\in\N$ there exists $z_n\in(x,y_n)$ such that
	$\|z_n\|=\min \{\|z\|: z\in[x,y_n]\}$. It is not difficult to see that
	$$\|z_n\|\geq\|x+y_n\|-1$$
	(indeed, if $z_n'\in(x,y_n)$ is such that $\frac{z_n+z_n'}2=\frac{x+y_n}2$, then
	$\|x+y_n\|=\|z_n+z_n'\|\le\|z_n\|+1$). For each $n\in\N$, let
	$f_n\in X^*$ be
	a norm-one functional that separates $\|z_n\|B_X$ and $[x,y_n]$;
	clearly, $$f_n(z_n)=\|z_n\|=f_n(x)=f_n(y_n).$$
	
	Notice that $\|z_n\|\to1$, that is, $f_n(x)\to1$. Since $x$ is a
	Fr\'echet smooth point of $B_X$, we have that $f_n\to f_x$ in the
	norm topology (see, e.g., \cite[Corollary~7.22]{FHHMZ}). It
	follows that $f_x(y_n)\to1$. In particular, $y_n$ belongs to
	$\Sigma$ for each sufficiently large $n$, and hence
	$\mathrm{diam}(\Sigma)\geq2$. This contradiction concludes the proof.
\end{proof}

\noindent By Proposition~\ref{P: suffcond}  and
Theorem~\ref{T: trebolle}, we obtain the following corollary.

\begin{corollary}\label{C: suffcond}
	Let $X$ be a Banach satisfying at least one of the following conditions:
	\begin{enumerate}
		\item $X$ is uniformly Fr\'echet smooth;
		\item there exists $\epsilon>0$ such that the set of all $\epsilon$-cone smooth points of $B_X$ is dense in $S_X$;
		\item $\K(X)$, the Kottman's constant of $X$, satisfies $\K(X)<2$;
		\item for each $x\in S_X$, at least one of the following conditions is satisfied:\begin{itemize}
			\item $x$ is an LUR point;
			\item $x$ is an LND2 point;
			\item $x$ is a Fr\'echet smooth and strongly exposed point of $B_X$;
			\item $x$ is a Fr\'echet smooth point and the unique norm-one functional $f_x\in X^*$ that supports $B_X$ at $x$ determines a slice $\Sigma$ of $B_X$
			with $\mathrm{diam}(\Sigma)<2$.
		\end{itemize}
	\end{enumerate}
	Then $X$ does not admit star-finite coverings by closed balls.
\end{corollary}

\smallskip

\subsection{Prohibitive conditions in spaces of continuous functions}  We shall use the following standard notation. Given a Hausdorff topological space $T$, by $C_b(T)$ we mean the Banach space of all bounded continuous real-valued functions on $T$, equipped with the supremum norm
$\|x\|_\infty:=\sup_{t\in T}|x(t)|$. In the case $T$ is compact, we simply write $C(T)$ instead of $C_b(T)$.
If $T$ is  a locally compact Hausdorff space, we denote by $C_0(T)$ the Banach space of all 
elements of $C_b(T)$ that vanish at infinity.

\begin{definition}
	Let $X$ be a normed space. We shall say that:
	\begin{enumerate}[(a)]
		\item a direction $v\in S_X$ is {\em important} if there exists $\alpha_v>0$ such that
		for each straight line $L\subset X$ which is parallel to $v$ and intersects $B_X$,
		one has 
		$\diam (L\cap B_X)\ge\alpha_v$;
		\item a point $x\in S_X$ is {\em ``good''} if  there exists an important direction $v\in S_X$ such that $\|x+tv\|>1$ for each $t>0$.
	\end{enumerate}
\end{definition}

\begin{theorem}\label{T: good pts}
	Let $X$ be an infinite-dimensional Banach space such that its ``good'' points are dense in $S_X$.
	Then $X$ has no countable star-finite covering by closed balls.
\end{theorem}

\begin{proof}
	Proceeding by contradiction, let $\mathcal B=\{B_n\}_{n\in\N}$ be a countable star-finite covering of $X$ by closed balls. We can (and do) assume that $\mathcal B$ is minimal, and denote 
	$$\textstyle
	D:=\bigcup_{n\in\N} \partial B_n\,,\quad
	H:=\{x\in D: |\mathcal{B}(x)|=1\}.
	$$
	By Observation~\ref{O: basic}(d), $H=\bigcup_{n\in\N}(\partial B_n\cap H)$ is a Baire space.
	Therefore there exists $m\in\N$ such that $\partial B_m\cap H$ is not nowhere dense in $H$.
	Using the fact that $\partial B_m\cap H$ is a relatively open set in $\partial B_m$ 
	(see Corollary~\ref{C: one ball}), it easily follows that there exist $x_0\in\partial B_m$ and $\epsilon>0$ so that
	\begin{equation}\label{E: UH}
	U(x_0,\epsilon)\cap H\subset \partial B_m\cap H.
	\end{equation}
	We can (and do) clearly assume that $B_m=B_X$ and $x_0$ is a ``good'' point. 
	Let $v\in S_X$ be an important direction such that the half-line
	$$
	L:=\{x_0+tv: t>0\}
	$$
	is disjoint from $B_X$. Notice that the subfamily 
	$\mathcal{B}':=\{B\in\mathcal{B}: B\cap L\ne\emptyset\}$ covers $L$,
	and $x_0\notin\bigcup\mathcal{B}'$ is necessarily a singular point for $\mathcal{B}'$. By star-finiteness, there exists an infinite disjoint subfamily $\mathcal{B}''\subset\mathcal{B}'$ whose elements are disjoint from $B_X$, and such that
	$$
	\inf_{B''\in\mathcal{B}''}d(x_0,B'')=0.
	$$
	Write $\mathcal{B}''=\{B''_k\}_{k\in\N}$ and notice that we can assume that $d(x_0,B''_k\cap L)\to0$. Then
	$\diam(B''_k\cap L)\to0$. Since the direction $v$ is important, we obtain that
	$$
	r(B''_k)\le(1/\alpha)\diam(B''_k\cap L)\to0.
	$$
	For each sufficiently large $k$, $B''_k\subset U(x_0,\epsilon)$, and since $B''_k$ is disjoint from $B_X=B_m$
	we obtain from \eqref{E: UH} that $B''_k\cap H=\emptyset$ for such $k$. But this contradicts Corollary~\ref{C: one ball}. We are done.
\end{proof}

\begin{theorem}
	Let $T$ be an infinite Hausdorff topological space whose isolated points form a dense subset. Let $X$ be a closed subspace of $C_b(T)$ such that $X$ contains the characteristic function 
	$\mathds{1}_{\{t\}}$ for each isolated point $t\in T$. Then the Banach space $X$ has no countable star-finite covering by closed balls.
\end{theorem}

\begin{proof}
	By Theorem~\ref{T: good pts} it suffices to show that ``good'' points of $X$ are dense in $S_X$.
	Fix arbitrary $x\in S_X$ and $\epsilon>0$. At least one of the open sets $\{t\in T: x(t)>1-\epsilon\}$
	and $\{t\in T: x(t)<-1+\epsilon\}$ is nonempty, say it is the first one (the other case is done in a similar way).
	So there exists an isolated point $t_0\in T$ such that $x(t_0)>1-\epsilon$.
	
	We claim that $v:=\mathds{1}_{\{t_0\}}\in S_X$ is an important direction for $X$. To this end, consider the line
	$L:=\{z+\lambda v: \lambda\in\R\}$ where $z\in X$, $\|z\|_\infty\le1$, and denote
	$$
	\beta:=\min\{\lambda\in\R: \|z+\lambda v\|_\infty\le1\}
	\ \text{and}\ 
	\gamma:=\max\{\lambda\in\R: \|z+\lambda v\|_\infty\le1\}.
	$$
	For each $\eta>0$ we have 
	$$
	1<\|z+(\beta-\eta)v\|_\infty=\max\left\{\sup_{t\ne t_0}|z(t)|\,,\,|z(t_0)+\beta-\eta|\right\}
	=|z(t_0)+\beta-\eta|,
	$$
	which implies that $z(t_0)+\beta=-1$. Analogously, we obtain that
	$|z(t_0)+\gamma+\eta|>1$ ($\eta>0$), and hence $z(t_0)+\gamma=1$. It follows that
	$\diam(L\cap B_X)=\gamma-\beta=2$, and our claim is proved.
	
	Now, by the choice of $t_0$, for each $\lambda\ge\epsilon$ we have
	$(x+\lambda v)(t_0)>(1-\epsilon)+\epsilon=1$ and hence $\|x+\lambda v\|_\infty>1$. Put
	$\lambda_0:=\max\{\lambda\ge0: \|x+\lambda v\|_\infty\le1\}$ and notice that $\lambda_0<\epsilon$ and
	$\|x+\lambda_0 v\|_\infty=1$. The point $y:=x+\lambda_0 v\in S_X$ is ``good'' since $v$ is an important direction
	and $\|y+\eta v\|_\infty=\|x+(\lambda_0+\eta)v\|_\infty>1$ ($\eta>0$).
	Moreover, $\|y-x\|_\infty=\lambda_0<\epsilon$.
	This completes the proof.
\end{proof}

\begin{corollary}\label{C: no countable}
	Let $T$ be an infinite Hausdorff topological space whose isolated points  are dense, and $\Gamma$ a nonempty infinite set.
	Let $X$ be one of the following spaces:
	\begin{enumerate}[(a)]
		\item $C(T)$ where $T$ is compact;
		\item $C_0(T)$ where $T$ is locally compact;
		\item $\ell_\infty(\Gamma)$ or $c_0(\Gamma)$.
	\end{enumerate}
	Then $X$ has no countable star-finite covering by closed balls.
\end{corollary}

The following result shows that $c_0(\Gamma)$ ($\Gamma$ infinite) has no (not necessarily countable)
star-finite covering by closed balls, and that
property $(\mathcal I_{\aleph_0})$ is only a sufficient condition for $X$ to not admit star-finite coverings by closed  balls.

\begin{theorem}\label{T: c_0dirbuona}
	Let $\Gamma$ be an infinite set. Then $c_0(\Gamma)$ does not admit any star-finite covering by closed balls,
	and it fails $(\mathcal I_{\aleph_0})$.	
\end{theorem}

\begin{proof}
Proceeding by contradiction, assume
that such a covering exists.
	Fix an infinite countable set $\Gamma_0\subset\Gamma$ and consider the separable subspace
	$Y:=\{x\in X: x(\gamma)=0 \text{ for each } \gamma\in\Gamma\setminus\Gamma_0\}$.
	The family $\mathcal{B}':=\{B\cap Y: B\in\mathcal{B}\ne\emptyset\}$ is a star-finite covering of $Y$.
	It is an easy exercise to see that each member of $\mathcal{B}'$ is in fact a closed ball in $Y$.
	By Observation~\ref{O: card upper bound}, $\mathcal{B}'$ is countable, but this contradicts
	Corollary~\ref{C: no countable}(c) since $Y$ is isometric to $c_0(\Gamma_0)$.
	For the second part, let the sequence $\{u_n\}\subset 2S_{c_0(\Gamma)}$ be defined by $$u_1=2e_1-e_2,\ \ \ u_n=2e_1+e_2+\ldots+e_n-e_{n+1}\ \ \ \ (n>1).$$ We claim that the point $x=e_1\in S_{c_0}$ fails property $(\mathcal I_{{\aleph_0},\epsilon})$, whenever $\epsilon>0$.
	Indeed, for each $\delta>0$, we can consider the family 
	$$\mathcal D:=\bigl\{(1+\delta)u_n+B_{c_0(\Gamma)};\,n\in\N\bigr\},$$ 
	and observe that $\mathcal D$ is a family of pairwise disjoint closed balls of radius 1 not intersecting $B_X$ and such that $|\mathcal D|={\aleph_0}$. Moreover, we have 
	$\mathrm{dist}(x,B)=2\delta$, whenever $B\in\mathcal D$. This clearly implies that $c_0(\Gamma)$ does not have property $(\mathcal I_{\aleph_0})$.
\end{proof}

\smallskip

\subsection{Simple coverings by closed balls}
Recall that a simple covering is a covering by pairwise disjoint sets.
It is a well-known fact that each simple covering of $\R$ by at least two nonempty closed subsets
of $\R$ is uncountable (see e.g.~\cite[Fact~3.2]{DEVETIL}). Hence if a (nontrivial) normed space $X$ admits a simple covering by closed balls, then necessarily $X$ is nonseparable and the covering is uncountable. Moreover, from this result we can easily deduce that certain non-separable $C(K)$ spaces do not admit simple covering by closed balls.

\begin{proposition}\label{P: C(K) no simple coverings}
	Let $K$ be a compact space. Suppose that $K$ contains an isolated point, then $C(K)$ does not admit  simple coverings by closed balls. 
\end{proposition}

\begin{proof}
	Let $k\in K$ be an isolated point, then the characteristic function $\mathds{1}_{\{k\}}$ is a continuous function on $K$. Let $B=B(f,r)$ be a closed ball in $C(K)$ intersecting the straight line $l=\{t\mathds{1}_{\{k\}}:t\in\R\}$. 
	We claim that $B\cap l$ is a non-degenerate closed interval. Indeed, since $B\cap l\neq \emptyset$, we have $|f(x)|\leq r$ for each $x\in K\setminus \{k\}$. It follows that
	$t\mathds{1}_{\{k\}}\in B$ if and only if $|t-f(x)|\le r$, proving our claim.
	Now it is clear that $C(K)$ can not be  covered by a simple family of closed balls since otherwise we would get a simple covering of $\R$ by non-degenerate closed intervals, which is impossible.
\end{proof}

Next theorem shows the relation between separated families of vectors and simple coverings by closed balls. For convenience of the reader we state the following known lemma.

\begin{lemma}\label{L: renorming}
	Let $(X,\|\cdot\|_X)$ and $(Y,\|\cdot\|_Y)$ be normed spaces  and $\theta\in[1,2]$. Let $A\subset S_X$ be such that  $\mathrm{sep}\, A\geq\theta$. Let $M\in[1,\theta]$ and let  $T:X\to Y$ be an isomorphic embedding  such that $\|T\|\cdot\|T^{-1}\|\leq M$.  Then the set $B:=\{\frac {Tx}{\|Tx\|_Y};\,x\in A\}\subset S_{Y}$ satisfies
	$$\textstyle \mathrm{sep}\, B\geq \frac\theta M.$$ 
\end{lemma}

\begin{proof}
	See the proof of \cite[Lemma~2.2]{HajekKaniaRusso}.
\end{proof}

\begin{theorem}\label{T: no simple stable}
	Let $Y$ be a Banach space such that $K:=\K(Y,\aleph_1)<2$. Let  $T\colon X\to Y$ be an isomorphic embedding  satisfying $\|T\|\cdot\|T^{-1}\|<\frac2K$.
	Then $X$ does not admit  any simple covering by closed balls.
\end{theorem}

\begin{proof}
	Denote $M:=\|T\|\cdot\|T^{-1}\|$. Proceeding by contradiction,
	suppose that $X$ admits a simple covering $\B$ by closed balls. 
	Let $\ell$ be a line in $X$, and observe that uncountably many elements of $\B$ intersect $\ell$. By Lemma~\ref{L: uncountable star.finite}, $X$ fails property $(\mathcal I_{\aleph_1})$ (and hence it fails property $(U\mathcal I_{\aleph_1})$). By Lemma~\ref{L: costante di Kottman}, we have $\K(X,\aleph_1)=2$ and hence there exists $A\subset S_X$ such that $\mathrm{sep}\, A> MK$ and   $|A|=\aleph_1$. Lemma~\ref{L: renorming} implies existence of a set $B\subset S_{Y}$ satisfying
	$\mathrm{sep}\, B> \frac{MK}{M}=K$ and $|B|=\aleph_1$. But this contradicts the definition of $K$.  
\end{proof}

\noindent A famous result of Elton and Odell \cite{EltonOdell} states that if 
$\Gamma$ is an uncountable set then  $c_0(\Gamma)$ contains no $(1+\varepsilon)$-separated uncountable family of unit vectors, for any $\varepsilon>0$. That is, $\K(c_0(\Gamma),\aleph_1)\leq1$ (observe that this inequality trivially holds even if $\Gamma$ is countable). Hence, we get the following corollary.

\begin{corollary}\label{C: no simple c0Gamma}
	Let $\Gamma$ be a nonempty set, and $X$ a Banach space. If there exists an isomorphic embedding $T\colon X\to c_0(\Gamma)$ such that $\|T\|\cdot\|T^{-1}\|<2$, then
	$X$ does not admit  simple coverings by closed balls.
\end{corollary}

Finally we observe that P. Koszmider in \cite{Kosz18} defined, under an additional set-theoretic assumption consistent with the usual axioms of ZFC, a connected compact space $K$ for which the Banach space $C(K)$ has no uncountable $(1+\varepsilon)$-separated set in the unit ball for any $\varepsilon>0$, hence $\K(C(K),\aleph_1)<2$. Therefore, by Theorem \ref{T: no simple stable}, $C(K)$ does not admit any simple covering by closed balls. 

\medskip

\subsection{Some open problems} We have already mentioned that separable nor\-med spaces do not admit a simple covering, however Theorem \ref{T: star-finite countable dimension} shows that normed spaces with countable dimension admit a star-finite covering by closed balls. On the other hand, in the present section we have provided various conditions for a Banach space not to have a star-finite covering by closed balls. Among them there is $c_0$ which admits a point-finite covering by closed balls. The following question naturally arises from these facts.
\begin{problem}
Does there exist a separable Banach space admitting a star-finite covering by closed balls?
\end{problem}

The following two problems should be compared with Corollary \ref{C: suffcond} and Corollary \ref{C: no countable}, respectively.

\begin{problem}
Does there exist an infinite-dimensional Fr\'echet smooth Banach space admitting a star-finite covering by
closed balls?
\end{problem}

\begin{problem}
Does there exist an infinite compact space $K$ for which $C(K)$ admits a star-finite (or even simple) covering by closed balls?
\end{problem}

\medskip

\section{A star-finite covering by Fr\'echet smooth bodies of $c_0(\Gamma)$}\label{section:existenceresult}

The purpose of this section is to show that, for any nonempty set $\Gamma$, the Banach space $c_0(\Gamma)$ admits a star-finite covering by Fr\'echet smooth bounded bodies. This is clearly trivial for any finite $\Gamma$; therefore, from now on $\Gamma$ will be an infinite set.

In order to define the desired bodies, we are going to define suitable Fr\'echet renormings of $c_0(\Gamma)$, whose balls, roughly speaking, have many flat faces.  Given $M>2$, let us consider the  equivalent norm on $c_0(\Gamma)$ defined for  $x\in c_0(\Gamma)$ by
$$\|x\|^2_M=\inf\{{\|x_1\|^2_\infty+M\|x_2\|^2_2};\, x_1\in c_0(\Gamma),x_2\in\ell_2(\Gamma), x_1+x_2=x\}.$$
Thanks to Proposition \ref{P: norma Frechet}, we have:
\begin{enumerate}
	\item $\|x\|_M\leq\|x\|_\infty\leq \sqrt{1+\frac{1}{M}}\,\|x\|_M$;
	\item the dual norm of $\|\cdot\|_M$ is given by:
	$$\textstyle {\|f\|}_M^*=\sqrt{\|f\|^2_1+\frac1M\|f\|^2_2}\ \ \ \ \ \ \ (f\in \ell_1(\Gamma)).$$
	\item $\|\cdot\|_M$ is Fr\'echet smooth (since its dual norm is LUR; this is quite standard);
	\item $\|\cdot\|_M$ is a lattice norm. 
\end{enumerate}
We will use $B_M$ to denote the closed unit ball of $(c_0(\Gamma), \|\cdot\|_M)$, and $B_{c_0(\Gamma)}$ to denote the one of $(c_0(\Gamma),\|\cdot\|_\infty)$. 
Observe that (i) is equivalent to 
 $$\textstyle{B_{c_0(\Gamma)}\subset B_M \subset \sqrt{1+\frac{1}{M}}\,B_{c_0(\Gamma)}}.$$
Let $\{e_\gamma\}_{\gamma\in\Gamma}$ be the canonical basis of $c_0(\Gamma)$ and,
for each finite set $\Gamma_0\subset \Gamma$, let us define
$$\textstyle Y_{\Gamma_0}={\mathrm{span}}\{e_\gamma;\, \gamma\in\Gamma_0\}\ \ \ \text{\and}\ \ \ Z_{\Gamma_0}=\overline{\mathrm{span}}\{e_\gamma;\, \gamma\in\Gamma\setminus\Gamma_0\}.$$
We denote by $P_{\Gamma_0}$ the canonical projection of $c_0(\Gamma)$ onto $Y_{\Gamma_0}$. Moreover, for $x\in c_0(\Gamma)$, we denote by $\mathrm{supp}(x)$ the support of $x$.

Let us start by quantifying how much flat is the norm $\|\cdot\|_M$. 

\begin{proposition}\label{P: bollapiatta}
	Suppose that $x\in c_0(\Gamma)$ is such that $\|x\|_M=1$ and let $y\in c_0(\Gamma)$ be such that:
	\begin{enumerate}
		\item[(a)] $\|y\|_\infty\leq1-\sqrt\frac2M\,$;
		\item[(b)] $\supp (x)\cap\supp(y)=\emptyset$. 
	\end{enumerate}
	Then $\|x+y\|_M=1$.
\end{proposition}

\begin{proof}
	Let $\epsilon\in(0,1)$ and let $x_1\in c_0(\Gamma)$ and $x_2\in\ell_2(\Gamma)$ be such that $x_1+x_2=x$, $\mathrm{supp}(x_1)\subset \mathrm{supp}(x)$ and $\|x_1\|_\infty^2+M\|x_2\|_2^2\leq1+\epsilon$. Observe that $\|x_2\|_2^2\leq\frac{1+\epsilon}M\leq\frac2M$ and hence that $\|x_2\|_\infty\leq\|x_2\|_2<\sqrt\frac2M\,$. Hence
	$$\textstyle \|x_1\|_\infty=\|x-x_2\|_\infty\geq\|x\|_\infty-\|x_2\|_\infty\geq1-\sqrt\frac2M\,.$$
	 By (a) and (b), it follows that $\|x_1+y\|_\infty=\|x_1\|_\infty$. Since $x+y=(x_1+y)+x_2$, we have that
	$$\|x+y\|_{M}^{2}\leq\|x_1+y\|^2_\infty+M\|x_2\|^2_2=\|x_1\|^2_\infty+M\|x_2\|^2_2\leq1+\epsilon.$$
	By arbitrariness of $\epsilon\in(0,1)$, we have that $\|x+y\|_M\leq1$. Moreover, by (b) and since $\|\cdot\|_M$ is a lattice norm, we clearly have $\|x+y\|_M=1$.
\end{proof}

Let $M>2$, $q\in(0,\infty)$ and $\Gamma_0\in [\Gamma]^{<\infty}$. 
Let us consider the continuous linear operator $T_{\Gamma_0,q}:c_0(\Gamma)\to c_0(\Gamma)$ given by $$\textstyle (T_{\Gamma_0,q}x)(\gamma)=\begin{cases}
\frac{x(\gamma)}q  &\text{if}\  \gamma\in\Gamma_0; \\
x(\gamma)  &\text{if}\   \gamma\in\Gamma\setminus\Gamma_0.
\end{cases}$$
Let us consider the  equivalent norm $\|\cdot\|_{M,\Gamma_0,q}$ on $c_0(\Gamma)$ given by
$\|x\|_{M,\Gamma_0,q}=\|T_{\Gamma_0,q} x\|_M$ ($x\in c_0(\Gamma)$). We observe that the mapping $T_{\Gamma_0,q}$ defines an isometry from $(c_0(\Gamma),\|\cdot\|_{M,\Gamma_0,q})$ onto $(c_0(\Gamma),\|\cdot\|_{M})$.
The following lemma easily follows by the definition of the norm $\|x\|_{M,\Gamma_0,q}$ and by Proposition~\ref{P: bollapiatta}.

\begin{lemma}\label{L: corpi}
	Let $\|\cdot\|_{M,\Gamma_0,q}$ be defined as above, and let $B_{M,\Gamma_0,q}$ be the corresponding unit ball. Then:
	\begin{enumerate}
		\item $B_{M,\Gamma_0,q}$ is a Fr\'echet smooth body;
		\item $q B_M\cap Y_{\Gamma_0}=B_{M,{\Gamma_0},q}\cap Y_{\Gamma_0}$;
		\item if $\Gamma_0\subset\Gamma_1\subset\Gamma$, $x\in B_{M,\Gamma_0,q}\cap Y_{\Gamma_1}$, and $y\in(1-\sqrt\frac2M\,)B_{c_0(\Gamma)}\cap Z_{\Gamma_1}$, then $x+y\in B_{M,\Gamma_0,q}$;
		\item $\|\cdot\|_{M,\Gamma_0,q}$ is a lattice norm;
		\item $B_{M,\Gamma_0,q}\subset q B_M\cap Y_{\Gamma_0}+\sqrt{1+\frac1M}\,B_{c_0(\Gamma)}\cap Z_{\Gamma_0}$.
	\end{enumerate}
\end{lemma}

\begin{proof}
	(i) It follows from the fact that the bijection $T_{\Gamma_0,q}$ is an isometry.\\
	(ii) If $x\in Y_{\Gamma_0}$ then $T_{\Gamma_0,q}(x)=\frac{x}{q}$. Therefore we have $\|T_{\Gamma_0,q}(x)\|_{M}=\|\frac{x}{q}\|_{M}\leq 1$ if and only if $\|x\|_{M}\leq q$.\\
	(iii) We start by proving the following assertion:
	\begin{equation}\label{E: terzo item}
	\begin{split}
	&\textstyle\text{if $\Gamma_0\subset \Gamma_1\subset \Gamma$, $x\in B_{M}\cap Y_{\Gamma_1}$, and $y\in(1-\sqrt\frac2M\,)B_{c_0(\Gamma)}\cap Z_{\Gamma_1}$,}\\ &\text{then $x+y\in B_{M}$.}
	\end{split}
\end{equation}
If $x=0$, then \eqref{E: terzo item} follows by the inclusion $B_{c_0(\Gamma)}\subset B_M$.	Let $x\in (B_M\cap Y_{\Gamma_1})\setminus\{0\}$ and $y\in(1-\sqrt\frac2M\,)B_{c_0(\Gamma)}\cap Z_{\Gamma_1}$. We have $\supp (\frac{x}{\|x\|_M})\cap\supp(y)=\emptyset$, hence by Proposition \ref{P: bollapiatta}, we obtain $\|\frac{x}{\|x\|_M}+y\|_{M}=1$. Since $\|\cdot\|_{M}$ is a lattice norm and $|x+y|\leq |\frac{x}{\|x\|_M}+y|$, we have $\|x+y\|_M\leq \|\frac{x}{\|x\|_M}+y\|_M=1$, hence \eqref{E: terzo item} is proved.\\
	Let $x\in B_{M,\Gamma_0,q}\cap Y_{\Gamma_1}$ and $y\in(1-\sqrt\frac2M\,)B_{c_0}\cap Z_{\Gamma_1}$. Since $T_{\Gamma_0,q}$ is an isometry and $x\in B_{M,\Gamma_0,q}\cap Y_{\Gamma_1}$, we have $T_{\Gamma_0,q}(x)\in B_{M}\cap Y_{\Gamma_1}$. Furthermore, since $y\in Z_{\Gamma_1}$, we have $T_{\Gamma_0,q}(y)=y$. Hence applying \eqref{E: terzo item} to $T_{\Gamma_0,q}(x)$ and $y$, we have $\|x+y\|_{M,\Gamma_0,q}=\|T_{\Gamma_0,q}(x)+T_{\Gamma_0,q}(y)\|_M=\|T_{\Gamma_0,q}(x)+y\|_M\leq 1$.\\
	(iv) 
	holds since $T_{\Gamma_0,q}$ is a positive operator, $\|\cdot\|_M$ is a lattice norm, and
	$\|\cdot\|_{M,\Gamma_0,q}=\|\cdot\|_M\circ T_{\Gamma_0,q}$.
	\\
	(v) Let $x\in B_{M,\Gamma_0,q}$. We set $x_1\coloneqq P_{\Gamma_0} (x)$ and $x_2\coloneqq (I-P_{\Gamma_0})(x)$. Let us prove that $x_1\in q B_M\cap Y_{\Gamma_0}$ and $x_2\in \sqrt{1+\frac1M}\,B_{c_0(\Gamma)}\cap Z_{\Gamma_0}$. Since $x\in B_{M,\Gamma_0,q}$, the norm $\|\cdot\|_{M,\Gamma_0,q}$ is a lattice norm and $|x_1|\leq |x|$, we have $x_1\in B_{M,\Gamma_0,q}$. Therefore by (ii), we have $x_1\in  qB_M\cap Y_{\Gamma_0}$. Since $x_2\in Z_{\Gamma_0}$, we have $T_{\Gamma_0,q}(x_2)=x_2$, therefore we obtain $\|x_2\|_{M}=\|T_{\Gamma_0,q}(x_2)\|_{M}=\|x_2\|_{M,\Gamma_0,q}\leq 1$. Finally, since $\|x\|_{\infty}\leq \sqrt{1+\frac1M}\, \|x\|_M$ holds for any $x\in c_0(\Gamma)$, we have $x_2\in \sqrt{1+\frac1M}\, B_{c_0(\Gamma)}\cap Z_{\Gamma_0}$. This completes the proof.
\end{proof}

\begin{theorem}
For every infinite set $\Gamma$, the space $c_0(\Gamma)$ admits a star-finite covering $\mathcal B$ by Fr\'echet smooth centrally symmetric bounded bodies.
\end{theorem}

\begin{proof}
Let us consider sequences $\{M_n\}_{n=0}^\infty\subset(2,\infty)$ and $\{\alpha_n\}_{n=0}^{\infty}\subset (0,1)$ such that $$\textstyle \theta:=\prod_{i=0}^{\infty}(1-\sqrt\frac2{M_i})\frac{\alpha_i}{\sqrt{1+M_i^{-1}}}>0.$$
Put $\theta_0=1$ and for each $n\in\N$ define 
$$\textstyle \theta_n:=\prod_{i=0}^{n-1}(1-\sqrt\frac2{M_i}\,)\,\prod_{j=1}^{n}\frac{\alpha_j}{\sqrt{1+M_j^{-1}}}\,.$$
We shall inductively construct families $\mathcal B_n$ ($n\in\N_0$) of bodies
such that:
\begin{enumerate}
	\item[($\mathrm P^1_n$)] if $B\in\mathcal B_n$ and $C\in\bigcup_{k<n}\mathcal B_k$, then $B\cap C=\emptyset$;
	\item[$(\mathrm P^2_n)$] $\mathcal B_n$ is star-finite;
	\item[($\mathrm P^3_n$)]
	$C^n:=\bigcup_{B\in \mathcal B_0\cup\ldots\cup \mathcal B_n} B$ is closed;
	\item[($\mathrm P^4_n$)]
	 for each $\Gamma_1\subset\Gamma$ such that $|\Gamma_1|\geq n$, we have
	$$\textstyle Y_{\Gamma_1}\cap C^n+\theta_n(1-\sqrt\frac2{M_n})B_{c_0(\Gamma)}\cap Z_{\Gamma_1}\subset C^n\subset Y_{\Gamma_1}\cap C^n+Z_{\Gamma_1};$$
    \item[($\mathrm P^5_n$)]
    for each $\Gamma_0\subset \Gamma$, with $|\Gamma_0|=n$, we have $$\textstyle Y_{\Gamma_0}+\theta_n \left( 1-\sqrt{\frac{2}{M_n}}\right)B_{c_0(\Gamma)}\subset C^n.$$
\end{enumerate} 
 Let us show that this is possible. 
 We put $\mathcal B_0=\{B_{M_0}\}$ and claim that the above conditions hold for $n=0$.
 Indeed, conditions ($\mathrm P^1_0$) and ($\mathrm P^2_0$)  are trivially true, while observing that $B_{c_0(\Gamma)}\subset B_{M_0}=C^0$, we obtain ($\mathrm P^3_0$) and ($\mathrm P^5_0$). In order to prove condition ($\mathrm P^4_0$), we verify both inclusions. By (iii) in Lemma \ref{L: corpi}, we have $Y_{\Gamma_1}\cap B_{M_0}+(1-\sqrt{\frac{2}{M_0}}\,)B_{c_0(\Gamma)}\cap Z_{\Gamma_1}\subset B_{M_0}$ for any $\Gamma_1\subset \Gamma$ such that $|\Gamma_1|\geq 0$. On the other hand, let $x\in B_{M_0}$ and $\Gamma_1\subset \Gamma$ such that $|\Gamma_1|\geq 0$. Since $P_{\Gamma_1}(x)\in Y_{\Gamma_1}$ and $|P_{\Gamma_1}(x)|\leq |x|$, we have $\|P_{\Gamma_1}(x)\|_{M_0}\leq\|x\|_{M_0}\leq 1$. Therefore it follows that $B_{M_0}\subset Y_{\Gamma_1}\cap B_{M_0}+Z_{\Gamma_1}$.  Hence ($\mathrm P^4_0$)  is established.
 
Let $n\in\N$ and suppose we have already defined  $\mathcal B_0,\ldots,\mathcal B_{n-1}$ 
 such that conditions ($\mathrm P^3_{n-1}$), ($\mathrm P^4_{n-1}$) and ($\mathrm P^5_{n-1}$) hold. Let $\Gamma_0\in [\Gamma]^n$.  We have that the set $C^{n-1}\cap Y_{\Gamma_0}$ is a closed subset of $Y_{\Gamma_0}$. By Lemma~\ref{L: totallybounded},  there exist sequences $\{x_k\}_k\subset Y_{\Gamma_0}$, and $\{\tilde q_k\}_k\subset(0,\infty)$ such that:
 \begin{enumerate}
 	\item[(a)] the family $\{x_k+\tilde q_k B_{M_n}\cap Y_{\Gamma_0}\}_k$ is star-finite;
 	\item[(b)] $\bigcup_{k}(x_k+\tilde q_k B_{M_n}\cap Y_{\Gamma_0})=Y_{\Gamma_0}\setminus C^{n-1}$;
 	\item[(c)] the singular points of $\{x_k+\tilde q_k B_{M_n}\cap Y_{\Gamma_0}\}_k$ are contained in $C^{n-1}\cap Y_{\Gamma_0}$. 
 \end{enumerate} 
Now, for each $k\in\N$, define  $q_k=\frac{\tilde q_k}{\theta_{n} }$ and put $\mathcal B_{\Gamma_0}=\{B_k\}_k$, where $B_k\coloneqq x_k+\theta_{n}  B_{M_n,\Gamma_0,q_k}$. Observe that, for each $k\in\N$, 
 \begin{equation}\label{E: uguaglianza}
 B_k\cap Y_{\Gamma_0}=x_k + \theta_{n} B_{M_n,\Gamma_0,q_k}\cap Y_{\Gamma_0}=x_k + \tilde q_k B_{M_n}\cap Y_{\Gamma_0}
 \end{equation}
 holds. Moreover, by (v) in Lemma \ref{L: corpi}, we have
  \begin{equation}\label{E: contenuto}
 \textstyle x_k + \theta_n B_{M_n,\Gamma_0,q_k}\subset x_k + \tilde q_k B_{M_n}\cap Y_{\Gamma_0} + \theta_n\sqrt{1+\frac{1}{M_n}}\,B_{c_0(\Gamma)}\cap Z_{\Gamma_0}
 \end{equation}
 for each $k\in\N$. Now, we are going to prove that the family $\B_{\Gamma_0}$ satisfies the following conditions:
\begin{enumerate}
	\item[(a')] the family $\mathcal B_{\Gamma_0}$ is star-finite;
	\item[(b')] $\bigcup_{B\in \mathcal B_{\Gamma_0}}B\cap Y_{\Gamma_0}=Y_{\Gamma_0}\setminus C^{n-1}$;
	\item[(c')] the singular points of $\mathcal B_{\Gamma_0}$ are contained in 
	$$\textstyle C^{n-1}\cap Y_{\Gamma_0}+\sqrt{1+M_n^{-1}}\,\theta_{n} B_{c_0(\Gamma)}\cap Z_{\Gamma_0}.$$ 
\end{enumerate}

If $\B_{\Gamma_0}$ is not star-finite, then there exists a subfamily $\{B_{k_j}\}_{j\in \N}\subset \B_{\Gamma_0}$ such that $B_{k_1}\cap B_{k_j}\neq \emptyset$, for each $j\in \N$. Let $y_j\in B_{k_1}\cap B_{k_j}$, for each $j\in \N$. By \eqref{E: contenuto}, for each $j\in \N$, we have  $$P_{\Gamma_0}(y_{j})\in [x_{k_j}+\tilde q_{k_j} B_{M_n}\cap Y_{\Gamma_{0}}]\cap[ x_{k_1}+\tilde q_{k_1} B_{M_n}\cap Y_{\Gamma_0}],$$  which contradicts $(a)$. Hence $(a')$ is proved.

$(b')$ follows combining \eqref{E: uguaglianza} with (b).

Let $x\in c_0(\Gamma)$ be a singular point for $\B_{\Gamma_0}$. Then $P_{\Gamma_0}(x)$ is a singular point for the family $\{x_k+\tilde q_k B_{M_n}\cap Y_{\Gamma_0}\}_k$, hence by (c), we have $P_{\Gamma_0}(x)\in C^{n-1}\cap Y_{\Gamma_0}$. Moreover we have $\|(I-P_{\Gamma_0})(x)\|_{\infty}\leq \sqrt{1+M_n^{-1}}\,\theta_{n}$ since otherwise, by \eqref{E: contenuto}, there would exist $\varepsilon>0$ for which $(x+ \varepsilon B_{c_0(\Gamma)})\cap B=\emptyset$ for each $B\in \B_{\Gamma_0}$, contradicting the fact that $x$ is singular. Therefore (c') is established.

Now, let us denote  
$$\textstyle
\B_n\coloneqq \bigcup_{\Gamma_0\in [\Gamma]^n}\B_{\Gamma_0}\,, 
\quad D^{n}_{\Gamma_0}:=\bigcup_{B\in \mathcal B_{\Gamma_0}}B
\quad\text{and}\quad D^n:=\bigcup_{\Gamma_0\in[\Gamma]^n}D_{\Gamma_0}^n. 
$$
\textbf{Claim:} there exists $\beta_n>0$ such that for every $B_0\in\B_{\Delta_0}$ and $B_1\in\B_{\Delta_1}$ with $\Delta_0,\Delta_1\in [\Gamma]^n$, $\Delta_0\neq \Delta_1$,  we have $\dist(B_0,B_1)\geq \beta_n$, where the distance  refers to the supremum norm.

In order to prove the claim, let $\Delta_0, \Delta_1\in [\Gamma]^n$ such that $\Delta_0\neq \Delta_1$ and $B_0\in \B_{\Delta_0}$, $B_1\in \B_{\Delta_1}$. Since $\Delta_0$ and $\Delta_1$ are different and they have the same cardinality, there exists $\gamma_0\in \Delta_0\setminus \Delta_1$. We observe that
$$B_0 \subset Y_{\Delta_0}\cap B_0 + Z_{\Delta_0}\subset Y_{\Delta_0}\cap B_0 + Z_{\{\gamma_0\}}. $$
Hence we have:
\begin{equation}\label{E: claim1}
\begin{split}
\dist(B_0,Y_{\Delta_1})&\geq \dist(Y_{\Delta_0}\cap B_0+Z_{\{\gamma_0\}},Y_{\Delta_1})\\
&=\inf\{\|x_0+z_0-y\|_{\infty}:x_0\in Y_{\Delta_0}\cap B_0,z_0\in Z_{\{\gamma_0\}},y\in Y_{\Delta_1}\}\\
&=\inf\{\|x_0+z_0\|_{\infty}:x_0\in Y_{\Delta_0}\cap B_0,z_0\in Z_{\{\gamma_0\}}\}\\
&\geq \inf\{|(x_0+z_0)(\gamma_0)|:x_0\in Y_{\Delta_0}\cap B_0,z_0\in Z_{\{\gamma_0\}}\}\\
&=\inf\{|x_0(\gamma_0)|:x_0\in Y_{\Delta_0}\cap B_0\}\\
&\geq \theta_{n-1}\left(\textstyle{1-\sqrt{\frac{2}{M_{n-1}}}}\right),
\end{split}
\end{equation}
where in the last inequality we have used property ($\mathrm{P_{n-1}^{5}}$) with $\Delta_0\setminus\{\gamma_0\}$. Moreover, by \eqref{E: contenuto} we have 
\begin{equation}\label{E: claim2}
\textstyle{B_1 \subset Y_{\Delta_1} + \theta_{n}\sqrt{1+\frac{1}{M_n}}B_{c_0(\Gamma)}\cap Z_{\Delta_1}}
\end{equation}
Hence, combining \eqref{E: claim1} with \eqref{E: claim2} we obtain

\begin{eqnarray*}
\dist(B_0,B_1)&\geq & \textstyle{\theta_{n-1} \left(1- \sqrt{\frac{2}{M_{n-1}}}\right) - \theta_n\sqrt{1+\frac{1}{M_n}}}\\
&=&\textstyle{\theta_{n-1} \left(1 - \sqrt{\frac{2}{M_{n-1}}}\right)(1-\alpha_n) >0.}
\end{eqnarray*}
Letting $\beta_n=\theta_{n-1} \left(1 - \sqrt{\frac{2}{M_{n-1}}}\right)(1-\alpha_n) >0$ we obtain the claim.

Let us prove that conditions ($\mathrm P^1_n$)-($\mathrm P^5_n$) hold.
\begin{itemize}
	\item In order to prove that condition ($\mathrm P^1_n$) holds, we can equivalently prove that the sets $C^{n-1}$ and $D^{n}_{\Gamma_0}$ are disjoint for each $\Gamma_0\in[\Gamma]^n$. Let $\Gamma_0 \in[\Gamma]^n$, $B\in\B_{\Gamma_0}$ and $x\in B$. By ($\mathrm P^4_{n-1}$) we have $C^{n-1}\subset Y_{\Gamma_0}\cap C^{n-1}+Z_{\Gamma_0}$. Therefore, suppose by contradiction that $x\in C^{n-1}$, then we would have $P_{\Gamma_0}(x)\in Y_{\Gamma_0}\cap C^{n-1}$. Which is not possible, indeed, by (b'), we have $P_{\Gamma_0}(x)\in B\cap Y_{\Gamma_0}\subset Y_{\Gamma_0}\setminus C^{n-1}$. 
	\item ($\mathrm P^2_{n}$) follows combining (a') with our claim.
	\item Let $\{z_{k}\}_{k\in\N}\subset D^n$ be such that $z_k\to z$. If there exists $B\in \B_n$ such that $z_k\in B$ for infinitely many $k\in \N$, by closedness of $B$, we have $z\in B\subset D^n$. If, on the other hand, each $B\in \B_n$ contains finitely many elements of the sequence $\{z_k\}_{k\in \N}$, by our claim, there exists $\Gamma_0\in[\Gamma]^n$ such that $z$ is a singular point of $\mathcal B_{\Gamma_0}$. By (c'), ($\mathrm P^4_{n-1}$) and the definition of $\theta_n$, we have
	$$\textstyle z\in C^{n-1}\cap Y_{\Gamma_0}+\theta_{n-1}(1-\sqrt\frac2{M_{n-1}}\,) B_{c_0(\Gamma)}\cap Z_{\Gamma_0}\subset C^{n-1}.$$
	In any case, the closure of the set $D^n$ is contained in $C^n=C^{n-1}\cup D^n$. Since, by ($\mathrm P^3_{n-1}$),  $C^{n-1}$ is closed, condition ($\mathrm P^3_{n}$) holds. 
	\item Since $\theta_n(1-\sqrt\frac2{M_n}\,)<\theta_{n-1}(1-\sqrt\frac2{M_{n-1}}\,)$ and since ($\mathrm P^4_{n-1}$) holds,
	in order to prove condition ($\mathrm P^4_{n}$), it suffices to show that, for each $\Gamma_1\subset \Gamma$ such that $|\Gamma_1|\geq n$, we have that 
	\begin{equation}\label{E: bollepiatte}
	\textstyle Y_{\Gamma_1}\cap D^{n}_{\Gamma_0}+\theta_n(1-\sqrt\frac2{M_n}\,)B_{c_0(\Gamma)}\cap Z_{\Gamma_1}\subset D^{n}_{\Gamma_0}\subset Y_{\Gamma_1}\cap D^{n}_{\Gamma_0}+Z_{\Gamma_1}
	\end{equation}
	for each $\Gamma_0\in[\Gamma]^n$. It is easy to see that \eqref{E: bollepiatte} follows by the definition of $\mathcal B_{\Gamma_0}$ and Lemma~\ref{L: corpi}, (iii) and (iv).
	\item By (b') we have $Y_{\Gamma_0} \subset C^{n}$. Since ($\mathrm P^4_{n}$) holds we have
	$$
	\textstyle Y_{\Gamma_0}\cap C^n+\theta_n(1-\sqrt\frac2{M_n}\,)B_{c_0(\Gamma)}\cap Z_{\Gamma_0}\subset C^n.$$
	Hence we obtain ($\mathrm P^5_{n}$).
\end{itemize} 
\smallskip

To complete the proof, let us consider the family $\mathcal B:=\bigcup_{n\in\N} \mathcal B_n$. By ($\mathrm P^1_{n}$) and ($\mathrm P^2_{n}$), $\mathcal B$ is clearly star-finite. Moreover, for each $n\geq 0$ and each $\Gamma_0\in[\Gamma]^{n}$, by condition ($\mathrm P^5_{n}$) we have that:
$$\textstyle Y_{\Gamma_0}+\theta B_{c_0(\Gamma)}\subset Y_{\Gamma_0}+\theta_n\left(1-\sqrt\frac2{M_n}\right)B_{c_0(\Gamma)}\subset C^n.$$
By arbitrariness of $n\geq 0$ and $\Gamma_0\in[\Gamma]^{n}$ (and since $\theta>0$), $\mathcal B$ is a covering of $c_0(\Gamma)$. The fact that the elements of $\mathcal B$ are Fr\'echet smooth centrally symmetric bounded bodies follows by our construction and Lemma~\ref{L: corpi}.
\end{proof}

\section{Appendix}\label{section:appendix}

In what follows, $(X,\|\cdot\|)$ and $(Y,|\cdot|)$ are Banach spaces whose dual norms will be denoted
by $\|\cdot\|_*$ and $|\cdot|_*$, respectively.

Given an arbitrary function $f\colon X\to(-\infty,+\infty]$ which is {\em proper}, that is, finite in at least one point, one can define its {\em Fenchel conjugate} $f^*\colon X^*\to (-\infty,+\infty]$ by 
$$
f^*(x^*)=\sup_{x\in X}\{x^*(x)-f(x)\}.
$$
Let us collect some useful properties, which are more or less known. 

\begin{lemma}\label{L: Fenchel}
Let $X,Y$ be as above, $f\colon X\to(-\infty,+\infty]$, $g\colon Y\to(-\infty,+\infty]$.
Let $\mathcal{C}$ denote the set of all convex, proper, lower semicontinuous functions on $X$
with values in $(-\infty,+\infty]$, and $\mathcal{C}^*$ the set of all convex, proper, weak$^*$-lower semicontinuous functions on $X^*$ with values in $(-\infty,+\infty]$. 
\begin{enumerate}[(a)]
\item $f^*$ is convex and weak$^*$-lower semicontinuous.
\item $f^*$ is proper if and only if $f\ge a$ for some continuous affine $a\colon X\to\R$. 
\item For any $\alpha>0$, $(\alpha f)^*(x^*)=\alpha\, f^*(x^*/\alpha)$, $x^*\in X^*$.
\item $(\|\cdot\|^2)^*=(1/4)\|\cdot\|_*^2$.
\item Let $T\colon Y\to X$ be a bounded linear operator, and assume that
$$
h(x):=\inf\bigl\{f(u)+g(y): u\in X, y\in Y, x=u+Ty\bigr\}>-\infty,\quad x\in X.
$$
Then the function $h$ is proper, and its Fenchel conjugate is $h^*=f^*+g^*\circ T^*$.
\item The Fenchel conjugation $\phi\mapsto \phi^*$ gives a bijection between $\mathcal{C}$ and $\mathcal{C}^*$.
\end{enumerate}
\end{lemma}

\begin{proof}[Sketch of proof]
(a), (b) and (c) are easy exercises. Part (d) can be easily proved via (b) from the known equality
$(\frac12\|\cdot\|^2)^*=\frac12\|\cdot\|_*^2$ (see \cite[Example~6.1.6]{Nic-Per} for a more general fact). 
Part (f) is a well-known result (sometimes called the Fenchel-Moreau theorem); see e.g.\ \cite[Proposition~4.4.2]{Bor-Van}, 
\cite[Theorem~6.1.2]{Nic-Per} or \cite[Theorem~5.2.8]{Lucc}. 

Let us show (e). In what follows, $x,u\in X$, $y\in Y$ and $x^*\in X^*$.
\begin{align*}
h^*(x^*)&= \sup_x\bigl\{x^*(x)-\!\!\inf_{x=u+Ty}[f(u)+g(y)]\bigr\}
=\sup_{u,y}\bigl\{ x^*(u+Ty)-f(u)-g(y)\bigr\}\\
&=\sup_u\{x^*(u)-f(u)\}+\sup_y\{(T^*x^*)(y)-g(y)\}=f^*(x^*)+g^*(T^* x^*).
\end{align*}
\end{proof}

\begin{proposition}\label{P: norma Frechet}
Let $X,Y$ be as above, $M>0$, and
$T\colon Y\to X$ a bounded linear operator. For $x\in X$ define $|\!|\!|x|\!|\!|\ge0$ by the formula
$$
|\!|\!|x|\!|\!|^2:=\inf\bigl\{\|u\|^2+M|y|^2: u\in X, y\in Y, x=u+Ty\bigr\}.
$$
Then:
\begin{enumerate}[(a)]
\item $|\!|\!|\cdot|\!|\!|$ is an equivalent norm on $X$ which satisfies the estimates
$$\textstyle
|\!|\!|x|\!|\!| \le \|x\| \le \sqrt{1+\frac{\|T\|^2}{M}}\  |\!|\!|x|\!|\!|\,;
$$
\item the corresponding dual norm is given by\ \ 
$|\!|\!|x^*|\!|\!|_*^2=\|x^*\|_*^2+\frac1M |T^* x^*|_*^2\,;$
\item if moreover $X,Y$ are Banach lattices and $T$ is a positive operator then $\norma{\cdot}$
is a lattice norm.
\end{enumerate}
\end{proposition}

\begin{proof}
It is easy to see that $|\!|\!|\cdot|\!|\!|>0$ outside the origin, $|\!|\!|\cdot|\!|\!|\le\|\cdot\|$,
and $|\!|\!|\lambda x|\!|\!| = |\lambda|  |\!|\!|x|\!|\!|$ whenever $\lambda\in\R$, $x\in X$.
Given $x_1,x_2\in X$ and $\epsilon>0$, for $i=1,2$ fix $u_i\in X$ and $y_i\in Y$ so that
$x_i=u_i+Ty_i$ and $\|u_i\|^2+M|Ty_i|^2\le \norma{x_i}^2+\epsilon$. Then clearly
$\norma{x_1+ x_2}^2\le\|u_1+u_2\|^2 +M|Ty_1+Ty_2|^2$, from which we obtain
\begin{align*}
\norma{x_1+x_2}&\le\sqrt{(\norm{u_1}+\norm{u_2})^2+\left(\sqrt{M}\,|Ty_1|+\sqrt{M}\,|Ty_2|\right)^2}\\
&\le \sqrt{\norm{u_1}^2+M|Ty_1|^2} + \sqrt{\norm{u_2}^2+M|Ty_2|^2}  \\
&\le\sqrt{\norma{x_1}^2+\epsilon}\, + \sqrt{\norma{x_2}^2+\epsilon}\,. 
\end{align*}
By $\epsilon\to0^+$ we obtain the triangle inequality for $\norma{\cdot}$. 
Consequently, $\norma{\cdot}$ is a norm on $X$, which is equivalent to $\|\cdot\|$ by the Open Mapping Theorem. Using Lemma~\ref{L: Fenchel}(c,d,e), it is not difficult to calculate that its dual norm
is given by $\norma{x^*}^2_*=\|x^*\|^2_*+(1/M)|T^*x^*|^2_*$. Thus
$\norma{\cdot}_* \le \sqrt{1+\frac{\|T\|^2}{M}}\,\norm{\cdot}_*$. It follows that
$\norm{\cdot}\le \sqrt{1+\frac{\|T\|^2}{M}}\,\norma{\cdot}$, which completes the proof of (a) and (b).

Now assume that $X,Y$ are Banach lattices and $T$ is positive. Then it is clear from (b) that
$(X^*,\norma{\cdot}_*)$ is a Banach lattice, and hence its dual $(X^{**},\norma{\cdot}_{**})$
is a Banach lattice as well.
Consequently $\norma{\cdot}$, which is the restriction of $\norma{\cdot}_{**}$ to $X$ (considered as a subspace of $X^{**}$), is a lattice norm.
\end{proof}

\end{document}